\documentclass[11pt,french,english,a4paper]{amsart}
\usepackage[utf8]{inputenc}
\usepackage[T1]{fontenc}
\usepackage[french]{babel}
\usepackage{verbatim}

\usepackage{amsmath}
\usepackage{amssymb}
\usepackage{amsthm}
\usepackage{lmodern}
\usepackage{latexsym}
\usepackage{amsfonts}
\usepackage{mathrsfs}
\usepackage[all]{xy}
\usepackage{bm}
\usepackage{yfonts}

\newcommand{\Ocal}{\mathcal{O}}

\usepackage[mathcal]{euscript}

\usepackage{hyperref}
\hypersetup{pdftoolbar=true, pdftitle={Proj_Ellip}, pdffitwindow=true, colorlinks=true, citecolor=green, filecolor=black, linkcolor=red, urlcolor=blue, hypertexnames=false}

\numberwithin{equation}{section}

\usepackage[top=4cm, bottom=3.4cm, left=3cm , right=3cm]{geometry}

\newtheorem{theorem}{Theorem}[section]

\newtheorem{proposition}[theorem]{Proposition}
\newtheorem{lemme}[theorem]{Lemma}

\usepackage{enumerate}

\usepackage{fge}

\theoremstyle{definition}

\newtheorem{remq}[theorem]{Remark}

\newcommand{\Cbb}{\mathbb{C}}

\newcommand{\Mscr}{\mathscr{M}}

\newcommand{\ifm}{\mathrm{if}}

\newcommand{\cond}{\mathrm{cond}}
\newcommand{\PGL}{\mathrm{PGL}}
\newcommand{\GL}{\mathrm{GL}}

\newcommand{\Brm}{\mathrm{\mathbf{B}}}
\newcommand{\Trm}{\mathrm{T}}
\newcommand{\Nrm}{\mathrm{\mathbf{N}}}

\newcommand{\Tr}{\mathrm{Tr}}

\newcommand{\Lcal}{\mathcal{L}}

\newcommand{\Zcal}{\mathcal{Z}}

\newcommand{\amf}{\textfrak{a}}

\newcommand{\Abf}{\mathbf{A}}
\newcommand{\omegabar}{\overline{\omega}}

\newcommand{\Abb}{\mathbb{A}}

\newcommand{\Cscr}{\mathscr{C}}

\newcommand{\pmf}{\mathfrak{p}}
\newcommand{\qmf}{\mathfrak{q}}
\newcommand{\varphibar}{\overline{\varphi}}

\newcommand{\Zrm}{\mathrm{\mathbf{Z}}}

\newcommand{\Prm}{\mathrm{\mathbf{P}}}
\newcommand{\Bscr}{\mathscr{B}}

\newcommand{\Nm}{\mathrm{\mathbf{N}}}

\newcommand{\Grm}{\mathrm{\mathbf{G}}}

\newcommand{\Zbb}{\mathbb{Z}}

\newcommand{\Xbf}{\mathbf{X}}
\newcommand{\Rbb}{\mathbb{R}}
\newcommand{\mmf}{\mathfrak{m}}
\newcommand{\Qbb}{\mathbb{Q}}

\newcommand{\F}{\mathrm{\mathbf{F}}}
\newcommand{\Krm}{\mathrm{\mathbf{K}}}

\newcommand{\Erm}{\mathrm{E}}
\newcommand{\Wcal}{\mathcal{W}}
\newcommand{\crm}{\mathrm{c}}
\newcommand{\pibar}{\overline{\pi}}

\newcommand{\Lscr}{\mathscr{L}}

\newcommand{\Wbf}{\mathrm{\mathbf{W}}}

\title{Periods and Reciprocity I}
\author{Raphaël Zacharias \\ \\  \'Ecole polythechnique f\'ed\'erale \\ de Lausanne}
\email{raphael.zacharias@epfl.ch}
\date{
    \today
}

\begin{document}

\maketitle

\begin{abstract}
Given $\F$ a number field with ring of integers $\Ocal_\F$ and $\pmf,\qmf$ two coprime and squarefree ideals of $\Ocal_\F$, we prove a reciprocity relation for the first moment of the triple product $L$-functions $L(\pi\otimes\pi_1\otimes\pi_2,\frac{1}{2})$ twisted by $\lambda_\pi(\pmf)$, where $\pi_1$ and $\pi_2$ are fixed unitary automorphic representations of $\PGL_2(\Abb_\F)$ with $\pi_1$ cuspidal and $\pi$ runs through unitary automorphic representations of conductor dividing $\qmf$. The method uses adelic integral representations of $L$-functions and the symmetric identity is established for a particular period. Finally, the integral period is connected to the moment via Parseval formula.
\end{abstract}

\tableofcontents

\section{Introduction}
Reciprocity formulae in the context of automorphic $L$-functions have been studied first for the second moment of Dirichlet $L$-functions $L(\chi,\frac{1}{2})$ modulo $q$ twisted by $\chi(p)$ by Conrey \cite{conrey}, Young \cite{young} and Bettin \cite{bettin}. In \cite{blomer}, Blomer and Khan considered a moment of the form $\sum_{f \ \mathrm{level \ } q}L(f\times F,s)L(f,w)\lambda_f(p)$, where $F$ is a fixed automorphic form for the group $\mathrm{SL}_3(\Zbb)$, $p,q$ are two prime numbers and $\lambda_f(p)$ is the $p^{\mathrm{th}}$ Hecke eigenvalue of $f$. Their result is non trivial even in the case $p=q=1$ since it also transforms the parameters $s$ and $w$. Their formula is a first contribution to the rich theory of automorphic forms on $\GL_4\times\GL_2$ over $\Qbb$ in the special situation where the $\GL_4$ function is an isobaric $3+1$-sum. We also mention the work of Blomer, Li and Miller \cite{blomer4} obtaining an identity involving the first moment of $L(\Pi\times f,\frac{1}{2})$ where this time $\Pi$ is a self-dual cusp form on $\GL_4$ and $f$ runs over $\GL_2$ modular forms.

\vspace{0.1cm}

In this paper, we provide an adelic treatment of the recent work of Andersen and Kiral \cite{andersen} who proved a reciprocity relation for the second moment of Rankin-Selberg $L$-functions $\sum_{f \ \mathrm{level} \ q}L(f\times g,\tfrac{1}{2})^2\lambda_f(p)\eta_f(q)$ with $g$ a fixed automorphic form for $\mathrm{SL}_2(\Zbb)$ and $\eta_f(q)=\pm 1$ is the eigenvalue of the Fricke involution. This is again a special case where the $\GL_4$ form is an isobaric $2+2$-sum, namely $g\boxplus g$. They obtained the result by means of classical tools in the analytic theory of automorphic forms; namely the Kuznetsov trace formula. At the end, the relation that exchanges the two prime numbers $p$ and $q$ is established by an identity of sums of Kloosterman sums. The main advantage in using the adelic language is that we can deal in a uniform way with arbitrary number fields and the symmetric relation appears in a very transparent way in some integral period (c.f. \eqref{DefinitionLcal}). 

\vspace{0.1cm}

Let  $\F$ be a number field. We write $\Ocal_\F$ for the ring of integers, $\Abb_\F$ for its adele ring and we let $\pmf,\qmf$ be two squarefree and coprime ideals of $\Ocal_\F$. We fix $\pi_1,\pi_2$ two unitary automorphic representations of $\PGL_2(\Abb_\F)$ which are unramified at all finite places and we assume that $\pi_1$ is cuspidal. We will prove a symmetric identity for the twisted first moment of the central value of the triple product $L$-functions, roughly of the shape (see Theorem \ref{Theorem} for the precise formulation)
\begin{equation}\label{Eq1}
\sum_{\substack{\pi \\ \cond(\pi)|\qmf}}L(\pi\otimes\pi_1\otimes\pi_2,\tfrac{1}{2})\lambda_\pi(\pmf)\eta_\pi(\qmf)\rightsquigarrow \sum_{\substack{\pi  \\ \cond(\pi)|\pmf}}L(\pi\otimes\pi_1\otimes\pi_2,\tfrac{1}{2})\lambda_\pi(\qmf)\eta_\pi(\pmf),
\end{equation}
where $\lambda_\pi(\pmf)$ is the eigenvalue of the Hecke operator $\Trm_\pmf$ and $\eta_\pi(\qmf)\in\{-1,+1\}$ denotes the Atkin-Lehner eigenvalue at $\qmf$.

\subsection{Some remarks}\label{SectionRemark}
\noindent 1) We recover the result of Andersen and Kiral if we choose $\pi_2$ to be the Eisenstein series $1\boxplus 1$ since in this case, $L(\pi\otimes\pi_1\otimes\pi_2,\tfrac{1}{2})=L(\pi\otimes\pi_1,\tfrac{1}{2})^2$.

\vspace{0.2cm}

\noindent 2) This method cannot be applied directly to the case where $\pi_1=\pi_2$ is Eisenstein, giving a reciprocity formula for a fourth moment of Hecke $L$-functions, because it is not $L^2$. However it is possible to handle this difficulty by using a regularized version of the inner product, valid for not necessarily square integrable automorphic forms, as developed by Michel and Venkatesh in \cite{subconvexity} and more recently by H. Wu \cite{Han}. This is the content of the forthcoming paper \cite{II}, in which we recover the same subconvex exponent as \cite[Theorem 4]{blomer}, namely $\tfrac{1}{4}-\tfrac{1-2\theta}{24}$, but extended to an arbitrary number field $\F$.

\vspace{0.2cm}

\noindent 3) We mention the work in progress by Ramon Nunes to give an adelic treatment of the paper \cite{blomer}, i.e. the case where $4$ decomposes as $3+1$.

\vspace{0.2cm}

\noindent 4) The method employed in this paper is very flexible and it should be possible to extend the result to higher rank automorphic forms, i.e. to $\GL_{2n}\times\GL_n$ at least when the form on $\GL_{2n}$ is an isobaric $n+n$-sum.

\subsection{Acknowledgment} I would like to thank my Phd supervisor Philippe Michel for his valuable advice. I am also grateful to the referee for his careful remarks and propositions who led to a significant simplification of the paper.

\section{Automorphic Preliminaries}

\subsection{Notations and conventions}\label{SectionNotations}

\subsubsection{Number fields} In this paper, $\F/\mathbb{Q}$ will denote a fixed number field with ring of intergers $\Ocal_\F$ and discriminant $\Delta_\F$. We make the assumption that all prime ideals considering in this paper do not divide $\Delta_\F$. We let $\Lambda_\F$ be the complete $\zeta$-function of $\F$; it has a simple pole at $s=1$ with residue $\Lambda_\F^*(1)$.

\subsubsection{Local fields} For $v$ a place of $\F$, we set $\F_v$ for the completion of $\F$ at the place $v$. We will also write $\F_\pmf$ if $v$ is finite place that corresponds to a prime ideal $\pmf$ of $\Ocal_\F$. If $v$ is non-Archimedean, we write $\Ocal_{\F_v}$ for the ring of integers in $\F_v$ with maximal $\mmf_v$ and uniformizer $\varpi_v$. The size of the residue field is $q_v=\Ocal_{\F_v}/\mmf_v$. For $s\in\Cbb$, we define the local zeta function $\zeta_{\F_v}(s)$ to be $(1-q_v^{-s})^{-1}$ if $v<\infty$, $\zeta_{\F_v}(s)=\pi^{-s/2}\Gamma(s/2)$ if $v$ is real and $\zeta_{\F_v}(s)=2(2\pi)^{-s}\Gamma(s)$ if $v$ is complex.

\subsubsection{Adele ring} The adele ring of $\F$ is denoted by $\Abb_\F$ and its unit group $\Abb^\times_\F$. We also set $\widehat{\Ocal_\F}=\prod_{v<\infty}\Ocal_{\F_v}$ for the profinite completion of $\Ocal_\F$ and $\Abb^1_\F=\{ x\in \Abb_\F^\times \ : \ |x|=1\}$, where $|\cdot| : \Abb_\F^\times \rightarrow \Rbb_{>0}$ is the adelic norm map.

\subsubsection{Additive characters} We denote by $\psi = \prod_v \psi_v$ the additive character $\psi_\Qbb\circ \Tr_{\F/\Qbb}$ where $\psi_\Qbb$ is the additive character on $\Qbb\setminus \Abb_\Qbb$ with value $e^{2\pi i x}$ on $\Rbb$. For $v<\infty$, we let $d_v$ be the conductor of $\psi_v$ : this is the smallest non-negative integer such that $\psi_v$ is trivial on $\mmf_v^{d_v}$. We have in this case $\Delta_\F=\prod_{v<\infty}q_v^{d_v}$. We also set $d_v=0$ for $v$ Archimedean.

\subsubsection{Subgroups} Let $\Grm=\GL_2$; if $R$ is a commutative ring, $\Grm(R)$ is by definition the group of $2\times 2$ matrices with coefficients in $R$ and determinant in $R^*$. We also defined the following standard subgroups
$$\Brm(R)=\left\{\left(\begin{matrix} a & b \\ & d\end{matrix}\right) \ : \ a,d\in R^*, b\in R\right\}, \ \Prm(R)= \left\{\left(\begin{matrix} a & b \\ & 1\end{matrix}\right) \ : \ a\in R^*, b\in R\right\}, \ $$
$$\Zrm(R)=\left\{\left(\begin{matrix} z & \\ & z\end{matrix}\right) \ : \ z\in R^*\right\}, \ \Abf(R)=\left\{\left(\begin{matrix} a &  \\ & 1\end{matrix}\right) \ : \ a\in R^*\right\}, $$ $$\Nm(R)=\left\{\left(\begin{matrix} 1 & b \\ & 1\end{matrix}\right) \ : \ b\in R\right\}.$$
We also set
$$n(x)=\left( \begin{matrix}
1 & x \\ & 1
\end{matrix}\right), \hspace{0.4cm}w = \left(\begin{matrix} & 1 \\ -1 & \end{matrix}\right) \hspace{0.4cm} \mathrm{and} \hspace{0.4cm} a(y)=\left( \begin{matrix}
y & \\ & 1
\end{matrix}\right).
$$
For any place $v$, we let $\Krm_v$ be the maximal compact subgroup of $\Grm(\F_v)$ defined by
\begin{equation}\label{Compact}
\Krm_v= \left\{ \begin{array}{lcl}
\Grm(\Ocal_{\F_v}) & \ifm & v \ \mathrm{is \ finite} \\
 & & \\
\mathrm{O}_2(\Rbb) & \ifm & v \ \mathrm{is \ real} \\
 & & \\
\mathrm{U}_2(\mathbb{C}) & \ifm & v \ \mathrm{is \ complex}.
\end{array}\right.
\end{equation}
We also set $\Krm:= \prod_v \Krm_v$. If $v<\infty$ and $n\geqslant 0$, we define 
$$\Krm_{v,0}(\varpi_v^n)= \left\{ \left( \begin{matrix}
a & b \\ c & d
\end{matrix}\right)\in \Krm_v \ : \  c\in \mmf_v^n\right\}.$$
If $\amf$ is an ideal of $\Ocal_\F$ with prime decomposition $\amf=\prod_{v<\infty}\pmf_v^{f_v(\amf)}$ ($\pmf_v$ is the prime ideal corresponding to the finite place $v$), then we set
$$\Krm_0(\amf)=\prod_{v<\infty} \Krm_{v,0}(\varpi_v^{f_v(\amf)}).$$

\subsubsection{Measures}\label{SectionMeasure} We use the same measures normalizations as in \cite[Sections 2.1,3.1]{subconvexity}. At each place $v$, $dx_v$ denotes a self-dual measure on $\F_v$ with respect to $\psi_v$. If $v<\infty$, $dx_v$ gives the measure $q_v^{-d_v/2}$ to $\Ocal_{\F_v}$. We define $dx=\prod_v dx_v$ on $\Abb_\F$. We take $d^\times x_v=\zeta_{\F_v}(1)\frac{dx_v}{|x_v|}$ as the Haar measure on the multiplicative group $\F_v^\times$ and $d^\times x = \prod_v d^\times x_v$ as the Haar measure on the idele group $\Abb^\times_\F$.

\vspace{0.1cm}

We provide $\Krm_v$ with the probability Haar measure $dk_v$. We identify the subgroups $\Zrm(\F_v)$, $\Nrm(\F_v)$ and $\Abf(\F_v)$ with respectively $\F_v^\times,$ $\F_v$ and $\F_v^\times$ and equipped them with the measure $d^\times z$, $dx_v$ and $d^\times y_v$. Using the Iwasawa decomposition, namely $\Grm(\F_v)=\Zrm(\F_v)\Nm(\F_v)\Abf(\F_v)\Krm_v$, a Haar measure on $\Grm(\F_v)$ is given by 
\begin{equation}\label{HaarMeasure}
dg_v = d^\times z dx_v \frac{d^\times y_v}{|y_v|}dk_v.
\end{equation} 
The measure on the adelic points of the various subgroups are just the product of the local measures defined above. We also denote by $dg$ the quotient measure on $$\Xbf:= \Zrm(\Abb_\F)\Grm(\F)\setminus \Grm(\Abb_\F),$$ 
with total mass $\mathrm{vol}(\Xbf)<\infty$.

\subsection{Automorphic forms and representations}\label{SectionAutomorphicRepresentations}
Let $L^2(\Xbf)$ be the Hilbert space of square integrable functions $\varphi : \Xbf\rightarrow \Cbb$. The $L^2$-norm is denoted by
\begin{equation}\label{L^2normCuspidal}
||\varphi||_{L^2(\Xbf)}^2 = \int_\Xbf |\varphi (g)|^2dg.
\end{equation}
We denote by $L_{\mathrm{cusp}}^2(\Xbf)$ the closed subspace of cusp forms, i.e. the functions $\varphi\in L^2(\Xbf)$ with the additional property that 
$$\int_{\F\setminus \Abb_\F}\varphi(n(x)g)dg=0, \ \ \mathrm{a.e.} \ g\in\Grm(\Abb_\F).$$
Each $\varphi\in L^2_{\mathrm{cusp}}(\Xbf)$ admits a Fourier expansion
\begin{equation}\label{FourierSeries}
\varphi(g)= \sum_{\alpha\in \F^\times}W_\varphi (a(\alpha)g),
\end{equation}
with
\begin{equation}\label{Whittaker-Cuspidal}
W_\varphi(g)=\int_{\F\setminus \Abb_\F}\varphi(n(x)g)\psi(-x)dx.
\end{equation}
The group $\Grm(\Abb_\F)$ acts by right translations on both spaces $L^2(\Xbf)$ and $L_{\mathrm{cusp}}^2(\Xbf)$ and the resulting representation is unitary with respect to \eqref{L^2normCuspidal}. It is well known that each irreducible component $\pi$ decomposes into $\pi = \otimes_v \pi_v$ where $\pi_v$ are irreducible and unitary representations of the local groups $\Grm(\F_v)$. The spectral decomposition is established in the first four chapters of \cite{analytic} and gives the orthogonal decomposition
\begin{equation}\label{OrthogonalDecomposition}
L^2(\Xbf)=L^2_{\mathrm{cusp}}(\Xbf)\oplus L^2_{\mathrm{res}}(\Xbf)\oplus L^2_{\mathrm{cont}}(\Xbf).
\end{equation}
$L^2_{\mathrm{cusp}}(\Xbf)$ decomposes as a direct sum of irreducible $\Grm(\Abb_\F)$-representations which are called the cuspidal automorphic representations. $L^2_{\mathrm{res}}(\Xbf)$ is the sum of all one dimensional subrepresentations of $L^2(\Xbf)$. Finally the continuous part $L^2_{\mathrm{cont}}(\Xbf)$ is a direct integral of irreducible $\Grm(\Abb_\F)$-representations and it is expressed via the Eisenstein series. In this paper, we call the irreducible components of $L^2_{\mathrm{cusp}}$ and $L^2_{\mathrm{cont}}$ the $\mathbf{unitary}$ automorphic representations. If $\pi$ is a unitary representation appearing in the continuous part, we say that $\pi$ is Eisenstein.

\vspace{0.1cm}

For any ideal $\amf$ of $\Ocal_\F$, we write $L^2(\Xbf,\amf):= L^2(\Xbf)^{\Krm_0(\amf)}$ for the subspace of level $\amf$ automorphic forms : this is the closed subspace of functions that are invariant under the subgroup $\Krm_0(\amf)$.

\subsection{Invariant inner product on the Whittaker model}\label{SectionInvariant} Let $\pi=\otimes_v\pi_v$ be a unitary automorphic representation of $\PGL_2(\Abb_\F)$ as defined in Section \ref{SectionAutomorphicRepresentations}. The intertwiner 
\begin{equation}\label{NaturalIntertwiner}
\pi\ni \varphi\longmapsto W_\varphi(g)=\int_{F\setminus \Abb}\varphi(n(x)g)\psi(-x)dx,
\end{equation}
realizes a $\GL_2(\Abb_\F)$-equivariant embedding of $\pi$ into a space of functions $W : \GL_2(\Abb_\F)\rightarrow \Cbb$ satisfying $W(n(x)g))=\psi(x)W(g)$. The image is called the Whittaker model of $\pi$ with respect to $\psi$ and it is denoted by $\Wcal(\pi,\psi)$. This space has a factorization $\otimes_v \Wcal(\pi_v,\psi_v)$ into local Whittaker models of $\pi_v$. A pure tensor $\otimes_v \varphi_v$ has a corresponding decomposition $\prod_v W_{\varphi_v}$ where $W_{\varphi_v}(1)=1$ and is $\Krm_v$-invariant for almost all place $v$. 

\vspace{0.1cm}

We define a normalized inner product on the space $\Wcal(\pi_v,\psi_v)$ by the rule 
\begin{equation}\label{NormalizedInnerProduct}
\vartheta_v(W_v,W_v') :=\zeta_{\F_v}(2)\frac{\int_{\F_v^\times}W_v(a(y))\overline{W}_v'(a(y))d^\times y}{\zeta_{\F_v}(1)L(\pi_v,\mathrm{Ad},1)}.
\end{equation}
This normalization has the good property that $\vartheta_v(W_v,W_v)=1$ for $\pi_v$ and $\psi_v$ unramified and $W_v(1)=1$ \cite[Proposition 2.3]{classification}. We also fix for each place $v$ an invariant inner product $\langle \cdot,\cdot\rangle_v$ on $\pi_v$ and an equivariant isometry $\pi_v \rightarrow \Wcal(\pi_v,\psi_v)$ with respect to \eqref{NormalizedInnerProduct}.

\vspace{0.1cm}

Recall that if $\pi$ is a cuspidal representation, we have a unitary structure on $\pi$ given by \eqref{L^2normCuspidal}. If $\pi$ belongs to the continuous spectrum and $\varphi$ is the Eisenstein series associated to a section $f : \GL_2(\Abb_\F)\rightarrow \Cbb$ in an induced representation of $\Brm(\Abb_\F)$ (see for example \cite[Section 4.1.6]{subconvexity} for the basic facts and notations concerning Eisenstein series), we can define the norm of $\varphi$ by setting 
$$||\varphi||^2_{\mathrm{Eis}}:= \int_{\Krm}|f(k)|^2dk.$$
We define the canonical norm of $\varphi$ by
\begin{equation}\label{CanonicalNorm}
||\varphi||^2_{\mathrm{can}} := \left\{ \begin{array}{lcl}
||\varphi||_{L^2(\Xbf)}^2 & \ifm & \pi \ \mathrm{is \ cuspidal} \\
 & & \\
2\Lambda_\F^*(1) ||\varphi||_{\mathrm{Eis}}^2 & \ifm & \pi \ \mathrm{is \ Eisenstein},
\end{array}\right.
\end{equation}
Using \cite[Lemma 2.2.3]{subconvexity}, we can compare the global and the local inner product : for $\varphi=\otimes_v \varphi_v \in \pi=\otimes_v\pi_v$ a pure tensor with $\pi$ either cuspidal or Eisenstein and non-singular, i.e. $\pi=\chi_1\boxplus\chi_2$ with $\chi_i$ unitary, $\chi_1\chi_2=1$ and $\chi_1\neq\chi_2,$ we have
\begin{equation}\label{Comparition}
||\varphi||_{\mathrm{can}}^2=2 \Delta_\F^{1/2} \Lambda^*(\pi,\mathrm{Ad},1)\prod_v \langle \varphi_{v},\varphi_v\rangle_v,
\end{equation}
where $\Lambda(\pi,\mathrm{Ad},s)$ is the complete adjoint $L$-function $\prod_v L(\pi,\mathrm{Ad},s)$ and $\Lambda^*(\pi,\mathrm{Ad},1)$ is the first nonvanishing coefficient in the Laurent expansion around $s=1$.

\subsection{Integral representations of triple product $L$-functions}\label{SectionRankin} 

Let $\pi_1,\pi_2,\pi_3$ be three unitary automorphic representations of $\PGL_2(\Abb_\F)$ such that at least one of them is cuspidal, say $\pi_2.$ Consider the linear functional on $\pi_1\otimes\pi_2\otimes\pi_3$ defined by
$$I (\varphi_1\otimes\varphi_2\otimes \varphi_3):= \int_\Xbf \varphi_1(g)\varphi_2(g)\varphi_3(g)dg.$$
This period is closely related to the central value of the triple product $L$-function $L(\pi_1\otimes\pi_2\otimes\pi_2,\tfrac{1}{2})$. In order the state the result, we write $\pi_i=\otimes_v \pi_{i,v}$ and for each $v$, we can consider the matrix coefficient
\begin{equation}\label{DefinitionMatrixCoefficient}
I'_v(\varphi_{1,v}\otimes\varphi_{2,v}\otimes\varphi_{3,v}) :=\int_{\PGL_2(\F_v)}\prod_{i=1}^3\langle \pi_{i,v}(g_v)\varphi_{i,v},\varphi_{i,v}\rangle_v dg_v. 
\end{equation}
It is a fact that \cite[(3.27)]{subconvexity} 
\begin{equation}\label{Fact}
\frac{I'(\varphi_{1,v}\otimes\varphi_{2,v}\otimes\varphi_{3,v})}{\prod_{i=1}^3 \langle \varphi_{i,v},\varphi_{i,v}\rangle_v}= \zeta_{\F_v}(2)^2 \frac{L(\pi_{1,v}\otimes\pi_{2,v}\otimes\pi_{3,v},\tfrac{1}{2})}{\prod_{i=1}^3 L(\pi_{i,v},\mathrm{Ad},1)},
\end{equation}
when $v$ is non-Archimedean and all vectors are unramified. It is therefore natural to consider the normalized version
\begin{equation}\label{DefinitionNormalizedMatrixCoefficient}
I_v(\varphi_{1,v}\otimes \varphi_{2,v}\otimes \varphi_{3,v}) := \zeta_{\F_v}(2)^{-2} \frac{\prod_{i=1}^3 L(\pi_{i,v},\mathrm{Ad},1)}{L(\pi_{1,v}\otimes\pi_{2,v}\otimes\pi_{3,v},\tfrac{1}{2})} I'_v (\varphi_{1,v}\otimes \varphi_{2,v},\varphi_{3,v}).
\end{equation}
The following proposition connects the global trilinear form $I$ with the central value $L(\pi_1\otimes\pi_2\otimes\pi_3,\tfrac{1}{2})$ and the local matrix coefficients $I_v$. The proof when at least one of the $\pi_i$'s is Eisenstein can be found in \cite[(4.21)]{subconvexity} and is a consequence of the Rankin-Selberg method. The result when all $\pi_i$ are cuspidal is due to Ichino \cite{ichino}.
\begin{proposition}\label{PropositionIntegralRepresentation} Let $\pi_1,\pi_2,\pi_3$ be unitary automorphic representations of $\PGL_2(\Abb_\F)$ such that at least one of them is cuspidal. Let $\varphi_i = \otimes_v \varphi_{i,v}\in \otimes_v \pi_{i,v}$ be pure tensors and set $\varphi :=\varphi_1\otimes\varphi_2\otimes\varphi_3$.
\begin{enumerate}[$a)$]
\item If none of the $\pi_i$'s is a singular Eisenstein series, then
\begin{equation}\label{Connection1}
\frac{|I(\varphi)|^2}{\prod_{i=1}^3 ||\varphi_i||^2_{\mathrm{can}}} = \frac{C}{8\Delta_\F^{3/2}}\cdot\frac{\Lambda(\pi_1\otimes\pi_2\otimes\pi_3,\tfrac{1}{2})}{\prod_{i=1}^3 \Lambda^*(\pi_i,\mathrm{Ad},1)}\prod_v \frac{I_v(\varphi_v)}{\prod_{i=1}^3\langle \varphi_{i,v},\varphi_{i,v}\rangle_v},
\end{equation}
where 
\begin{equation}\label{Constant}
C=\left\{ \begin{array}{ccl}
\Lambda_\F(2) & \ifm & \mathrm{all \ }\pi_i \ \mathrm{are \ cuspidal} \\ 
& & \\
 1 & \ifm & \mathrm{at \ least \ one \ is \ Eisenstein \ and \ nonsingular.}
\end{array}\right.
\end{equation}
\item
Assume that $\pi_3=1\boxplus 1$ and let $\varphi_3$ be the Eisenstein associated to the section $f_3(0)\in 1\boxplus 1$ which, for $\Re e (s)>0$, is defined by
$$f_3(g,s):= |\det(g)|^s\int_{\Abb_\F^\times} \Phi((0,t)g)|t|^{1+2s}d^\times t \in |\cdot|^s\boxplus |\cdot|^{-s},$$
where $\Phi=\prod_v \Phi_v$ with $\Phi_v = \mathbf{1}_{\Ocal_{\F_v}}^2$ for finite $v$. Then
\begin{equation}\label{Connection2}
\frac{|I(\varphi)|^2}{\prod_{i=1}^2 ||\varphi_i||^2_{\mathrm{can}}} = \frac{1}{4\Delta_\F}\cdot \frac{\Lambda(\pi_1\otimes\pi_2\otimes\pi_3,\tfrac{1}{2})}{\prod_{i=1}^2\Lambda^*(\pi_i,\mathrm{Ad},1)}\prod_v \frac{I_v(\varphi_{1,v}\otimes\varphi_{2,v}\otimes f_{3,v})}{\prod_{i=1}^2\langle \varphi_{i,v},\varphi_{i,v}\rangle_v}.
\end{equation}
\end{enumerate}
\end{proposition}

\subsection{The spectral decomposition}\label{SectionSpectral}
\subsubsection{The Parseval formula}\label{SectionParseval} We state here a version of Parseval Formula with level restriction. Let $\amf$ be an ideal of $\Ocal_\F$. We start with the orthogonal decomposition which is a direct consequence of \eqref{OrthogonalDecomposition}
$$L^2(\Xbf,\amf)=L^{2}_{\mathrm{cusp}}(\Xbf,\amf)\oplus L^{2}_{\mathrm{res}}(\Xbf,\amf)\oplus L^2_{\mathrm{cont}}(\Xbf,\amf).$$
From the classical orthogonal decomposition of the space of cusp forms into a direct sum, we obtain
$$L_{\mathrm{cusp}}^2(\Xbf,\amf):= L_{\mathrm{cusp}}^2(\Xbf)^{\Krm_0(\amf)}=\bigoplus_{\pi \ \mathrm{cuspidal}}\pi^{\Krm_0(\amf)}.$$
Observe that the cuspidal representations $\pi$ with $\pi^{\Krm_0(\amf)}\neq 0$ are those whose conductor (in the sense of \cite[Section 3.1.2]{cogdell}) divides $\amf$. For each such $\pi$, we let $\Bscr(\pi,\amf)$ be an orthonormal basis of $\pi^{\Krm_0(\amf)}$. The space $L^2_{\mathrm{res}}(\Xbf,\amf)$ is the orthogonal direct sum \cite[Proposition 6.1]{kuznetsov}
$$\bigoplus_{\omega^2=1}\Cbb \cdot \phi_\omega, \ \ \phi_\omega(g):= \omega(\det(g)).$$
The parametrization of the continuous part $L^{2}_{\mathrm{cont}}(\Xbf,\amf)$ passes by the induced representations $\omega|\cdot|^{it}\boxplus\overline{\omega}|\cdot|^{-it}$ of $\Brm(\Abb_\F)$ 
for $\omega$ a character of $\F^\times\setminus \Abb_\F^{1}$ with conductor satisfying $\cond(\omega)^2 | \amf$, $t\in\Rbb$ and it is made through the Eisenstein series. For such a $\omega$, we let $\Bscr(\omega,\amf)$ be an orthonormal basis of the space of $\Krm_0(\amf)$-invariant vectors in $\omega\boxplus\overline{\omega}$. It is then possible, using a $\Krm$-equivariant isomorphism $\omega\boxplus\omegabar \rightarrow \omega|\cdot|^{it}\boxplus\omegabar|\cdot|^{-it}$ \cite[Section 5]{kuznetsov}, to obtain a basis $\Bscr_{it}(\omega,\amf)$ of $\Krm_0(\amf)$-invariant vectors in $\omega|\cdot|^{it}\boxplus\overline{\omega}|\cdot|^{-it}$ for every $t\in\Rbb$. Then the Parseval formula holds \cite[Theorem 6.2]{kuznetsov}.
\begin{proposition}\label{PropositionParseval} Let $f,g\in L^2(\Xbf,\amf)$ and put $c:=\mathrm{vol}(\Xbf)^{-1}>0$. Then
\begin{alignat*}{1}
\langle f,g\rangle = & \ \sum_{\substack{\pi \ \mathrm{cuspidal} \\ \cond(\pi) |  \amf}}\sum_{\psi \in \Bscr(\pi,\amf)}\langle f,\psi\rangle \overline{\langle g,\psi\rangle}+c\sum_{\omega^2=1}\langle f,\phi_\omega\rangle\overline{\langle g,\phi_\omega \rangle} \\ 
 + & \ \frac{1}{4\pi}\sum_{\substack{\omega \in \widehat{\F^\times\setminus\Abb_\F^{1}} \\ \cond(\omega)^2 | \amf}}\int_{-\infty}^{+\infty}\sum_{\phi_{it}\in\Bscr_{it}(\omega,\amf)}\langle f,\Erm(\cdot,\phi_{it})\rangle\overline{\langle g,\Erm(\cdot,\phi_{it})\rangle}dt.
\end{alignat*}
\end{proposition}

\subsubsection{Hecke operators}\label{SectionHecke} It is enough for our purpose to restrict ourselves to Hecke operators of prime level because the Hecke operators are multiplicative and only squarefree ideals are considered in this paper. Let $\pmf\in\mathrm{Spec}(\Ocal_\F)$ and let $p:= |\Ocal_{\F_\pmf}/(\varpi_\pmf)|$ be the size of the residue field. We define the normalized Hecke operator $\Trm_\pmf$ as follows : for $f\in \mathscr{C}^\infty(\Grm(\Abb_\F))$, we set
\begin{equation}\label{DefinitionHeckeOperator}
\Trm_\pmf(f)(g):= \frac{1}{p^{1/2}}\int_{H_\pmf}f(gh)dh,
\end{equation}
where $H_\pmf$ is the double coset
\begin{equation}\label{DoubleCoset}
H_\pmf = \Krm_\pmf \left( \begin{matrix}
\varpi_\pmf & \\ & 1
\end{matrix}\right) \Krm_\pmf\subset \Grm(\F_{\pmf}),
\end{equation}
and the function $h\in H_\pmf\mapsto f(gh)$ has to be understood under the inclusion $\Grm(\F_\pmf)\hookrightarrow \Grm(\Abb_\F)$, $x\mapsto (1,...,1\underbrace{,x,}_{v=\pmf}1,...)$.
\section{Periods}
Through this section, we fix $\pi_1,\pi_2$ two unitary automorphic representation of $\PGL_2(\Abb_\F)$ with the properties that $\pi_1$ is cuspidal, both are unramified at all finite places and $\pi_2$ is not isomorphic to a quadratic twist $\pi_1\otimes\omega$. By multiplicativity of the Hecke operators and the fact that all manipulations are made \og place by place \fg, there is no loss in generality in restricting  to prime ideals. We thus fix $\pmf,\qmf$ two distinct elements of $\mathrm{Spec}(\Ocal_\F)$ of respective norm $p$ and $q$. We also assume for simplicity that $\pmf$ and $\qmf$ do not divide the discriminant $\Delta_\F$ of $\F.$

\subsection{A symmetric map}  We consider the $\GL_2(\Abb_\F)$-invariant map $\mathscr{P}_{\qmf,\pmf} : \pi_1\otimes\pi_2\longrightarrow \Cbb$ given by
\begin{equation}\label{DefinitionLcal}
\mathscr{P}_{\qmf,\pmf}(\varphi_1,\varphi_2):= \int_{\Xbf}\varphi_1^{\qmf\pmf}\varphi_2^{\pmf}\varphibar_1\varphibar_2^{\qmf},
\end{equation}
where we write $\varphi^\qmf$ for the right translate of $\varphi$ by the matrix $\left(\begin{smallmatrix} 1 & \\ & \varpi_\qmf\end{smallmatrix} \right)\in\GL_2(\F_\qmf)$ embedded in $\GL_2(\Abb_\F)$ as in § \ref{SectionHecke} and $\varpi_\qmf$ is a uniformizer of the local ring $\Ocal_{\F_\qmf}$. Observe that since we are dealing with $\PGL_2(\Abb_\F)$-representations we have $\pi_i=\pibar_i$ and thus, we get the symmetric relation (in terms of $\pmf$ and $\qmf$)
\begin{equation}\label{Symmetry}
\mathscr{P}_{\qmf,\pmf}(\varphi_1,\varphi_2)=\mathscr{P}_{\pmf,\qmf}(\varphi_1,\varphibar_2).
\end{equation}
This map will be connected via Parseval formula to a moment of automorphic $L$-functions (see Section \ref{SectionApplyParseval}).

\subsection{Invariant trilinear functionals} Let $\pi$ be a unitary automorphic representation of $\PGL_2(\Abb_\F)$ with finite conductor $\qmf$ and let $\pi_1,\pi_2$ as above. We have the following easy but very usefull proposition.
\begin{proposition}\label{PropositionTrilinear} Let $\ell : \pi\otimes\pi_1\otimes\pi_2\longrightarrow \Cbb$ be a $\GL_2(\Abb_\F)$-invariant trilinear functional and let $\varphi_i\in\pi_i$, $i=1,2,$ be vectors which are invariant under the subgroups $\Krm_\pmf$ and $\Krm_\qmf$. We let $\varphi$ be a $\Krm_0(\qmf)$-invariant vector in $\pi$. Then 
\begin{equation}\label{PropHecke}
\frac{p+1}{p^{1/2}}\ell(\varphi,\varphi_1^{\qmf\pmf},\varphi_2^{\pmf})=\lambda_\pi(\pmf)\ell(\varphi,\varphi_1^{\qmf},\varphi_2)
\end{equation}
and
\begin{equation}\label{PropAtkin}
\ell(\varphi,\varphi_1,\varphi_2^\qmf)=\eta_\pi(\qmf)\ell(\varphi,\varphi_1^\qmf,\varphi_2),\end{equation}
where $\lambda_\pi(\pmf)\in\Rbb$ is the eigenvalue of $\Trm_\pmf$ and $\eta_\pi(\qmf)\in\{-1,+1\}$ is the eigenvalue of the Atkin-Lehner operator.
\end{proposition}
\begin{proof}
Writing $\varphi^h$ for the right translate of $\varphi$ by $h\in H_\pmf$ (c.f. \eqref{DoubleCoset}), we have
\begin{alignat*}{1}
\lambda_\pi(\pmf)\ell\left(\varphi,\varphi_1^\qmf,\varphi_2\right)=\ell\left(\Trm_\pmf(\varphi),\varphi_1^\qmf,\varphi_2\right)=\frac{1}{p^{1/2}}\int_{H_\pmf}\ell\left( \varphi^h,\varphi_1^\qmf,\varphi_2\right) dh.
\end{alignat*}
Using successively the right $\Krm_\pmf$-invariance of $\varphi,$ the $\GL_2(\Abb_\F)$-invariance of $\ell$, the $\Zrm(\Abb_\F)$ and $\Krm_\pmf$-invariance of $\varphi_1,\varphi_2$, we see that for any $h\in H_\pmf$, we have
$$\left( \varphi^h,\varphi_1^\qmf,\varphi_2\right) =\ell\left(\varphi,\varphi_1^{\pmf\qmf},\varphi_2^{\pmf}\right),$$
which together with $\deg(\Trm_\pmf)=p+1$ implies the first assertion.

Finally, using the fact that $\varphi$ is $\Krm_0(\qmf)$-invariant, we have $\left(\begin{smallmatrix} & -1 \\ \varpi_\qmf & \end{smallmatrix}\right)\cdot\varphi = \eta_\pi(\qmf)\varphi$ and the $\Krm_\qmf$-invariance of $\varphi_1$ gives
\begin{alignat*}{1}
\ell(\varphi,\varphi_1,\varphi_2^\qmf)= & \ \ell\left(\left(\begin{smallmatrix} \varpi_\qmf & \\ & 1 \end{smallmatrix}\right)\cdot\varphi,\left(\begin{smallmatrix} \varpi_\qmf & \\ & 1\end{smallmatrix} \right)\cdot\varphi_1,\varphi_2\right) \\ 
 = & \ \ell\left(w\left(\begin{smallmatrix} \varpi_\qmf & \\ & 1 \end{smallmatrix}\right)\cdot\varphi,w\left(\begin{smallmatrix} \varpi_\qmf & \\ & 1\end{smallmatrix} \right)w\cdot\varphi_1,\varphi_2\right) \\
  = & \ \ell \left(\left(\begin{smallmatrix} & -1 \\ \varpi_\qmf & \end{smallmatrix}\right)\cdot\varphi, \varphi_1^{\qmf},\varphi_2\right) = \eta_\pi(\qmf)\ell(\varphi,\varphi_1^\qmf,\varphi_2),
\end{alignat*}
which completes the proof of this proposition.
\end{proof}

\subsection{Applying Parseval formula}\label{SectionApplyParseval} Let $\varphi_i=\otimes_v\varphi_{i,v}\in\pi_i=\otimes_{v}\pi_{i,v}$ be $\GL_2(\widehat{\Ocal}_\F)$-invariant vectors. We assume that $\langle \varphi_{i,v},\varphi_{i,v}\rangle_v=1$ for all $v$ and $i=1,2$ if $\pi_2$ is not $1\boxplus 1$. If $\pi_2=1\boxplus 1$, we choose $\varphi_3$ to be the Eisenstein series appearing in Proposition \ref{PropositionIntegralRepresentation} b). Applying Parseval formula \ref{PropositionParseval} to the period \eqref{DefinitionLcal} in the space $L^2(\Xbf,\qmf)$ (since $\varphi_1\varphi_2^{\qmf}$ is $\Krm_0(\qmf)$-invariant \footnote{In fact, Proposition \ref{PropositionParseval} asserts that we have to decompose our forms in the space $L^2(\Xbf,\pmf\qmf)$ but using the fact this space is the direct sum of $L^2(\Xbf,\qmf)$ with its orthogonal complement, we can restrict to the subspace $L^2(\Xbf,\qmf)$}) yields the spectral decomposition
\begin{equation}
\begin{split}
\mathscr{P}_{\qmf,\pmf}(\varphi_1,\varphi_2)= & \ \sum_{\substack{\pi \ \mathrm{cuspidal} \\ \crm(\pi) | \qmf}}\sum_{\psi\in\Bscr(\pi,\qmf)}\langle \varphi_1^{\pmf\qmf}\varphi_2^\pmf,\psi\rangle\overline{\langle\varphi_1\varphi_2^{\qmf},\psi\rangle} \\ + & \ \frac{1}{\mathrm{vol}(\Xbf)}\sum_{\omega^2=1}\langle \varphi_1^{\pmf\qmf}\varphi_2^{\pmf},\phi_\omega\rangle\overline{\langle \varphi_1\varphi_2^{\qmf},\phi_\omega \rangle} \\ 
 + & \ \frac{1}{4\pi}\sum_{\substack{\omega \in \widehat{\F^\times\setminus\Abb_\F^{1}} \\ \crm(\omega)^2 | \qmf}}\int_{-\infty}^{+\infty}\sum_{\phi_{it}\in\Bscr_{it}(\omega,\qmf)}\langle \varphi_1^{\pmf\qmf}\varphi_2^{\pmf},\Erm(\cdot,\phi_{it})\rangle\overline{\langle \varphi_1\varphi_2^{\qmf},\Erm(\cdot,\phi_{it})\rangle}dt.
\end{split}
\end{equation}
We now connect the expansion of $\mathscr{P}_{\pmf,\qmf}(\varphi_1,\varphi_2)$ to a moment of $L$-functions. We first observe for any $\omega$ with $\omega^2=1$, $\phi_\omega\varphi_1$ is a cusp form lying in the cuspidal representation $\pi\otimes\omega$. By hypothesis on $\pi_2$, namely that it is not isomorphic to a quadratic twist of $\pi_1$, we see that the contribution of the one-dimensional part is zero. Of course if $\pi_2$ is Eisenstein, this is classical orthogonality result between cusp forms and Eisenstein series.

\vspace{0.1cm}

For the two other parts, let $\pi$ be either cuspidal with $\crm(\pi)|\qmf$ or a principal series of the form $\omega|\cdot|^{it}\boxplus \omega^{-1}|\cdot|^{-it}$ with $\crm(\omega)=\Ocal_\F$ and set $\Bscr(\pi,\qmf)$ for an orthonormal basis of the space $\pi^{\Krm_0(\qmf)}$. We can of course assume that $t\neq 0$ since $t=0$ does not contribute in the integral. Using Proposition \ref{PropositionTrilinear} with the invariant trilinear form given by the inner product, we obtain (The Hecke eigenvalues are real because the central character is trivial)
$$\sum_{\psi\in\Bscr(\pi,\qmf)}\langle \varphi_1^{\pmf\qmf}\varphi_2^\pmf,\psi\rangle\overline{\langle\varphi_1\varphi_2^{\qmf},\psi\rangle}=\frac{p^{1/2}}{p+1}\lambda_\pi(\pmf)\sum_{\psi\in\Bscr(\pi,\qmf)}\langle \varphi_1^{\qmf}\varphi_2,\psi\rangle\overline{\langle\varphi_1\varphi_2^{\qmf},\psi\rangle}.$$
Here $\psi$ is understood to be an Eisenstein series if $\pi$ is not cuspidal. Define
\begin{equation}\label{DefL}
\Lscr(\pi,\qmf):= \sum_{\psi\in\Bscr(\pi,\qmf)}\langle \varphi_1^{\qmf}\varphi_2,\psi\rangle\overline{\langle\varphi_1\varphi_2^{\qmf},\psi\rangle}.
\end{equation}
We have the following proposition.
\begin{proposition}\label{PropositionFact3} Let $\Lscr(\pi,\qmf)$ be defined in \eqref{DefL}. Then
\begin{equation}\label{Fact3}
\Lscr(\pi,\qmf)= \left\{ \begin{array}{lcl}
\eta_\pi(\qmf)\sum_{\psi\in\Bscr(\pi,\qmf)}|\langle \varphi_1^\qmf\varphi_2,\psi\rangle |^2 & \ifm & \crm(\pi)=\qmf \\
 & & \\
\kappa \sum_{\psi \in \Bscr(\pi,\Ocal_\F)}|\langle \varphi_1\varphi_2,\psi\rangle |^2 & \ifm & \crm(\pi)=\Ocal_\F,
\end{array}\right.
\end{equation}
where $\kappa$ is the explicit constant given in \eqref{ValueKappa} which depends on the local datas $\pi_v$ and $\pi_{i,v}$ at the place corresponding to $\qmf$.
\end{proposition}
\begin{proof}
If the conductor of $\pi$ is $\qmf$, then this is consequence of \eqref{PropAtkin} in Proposition \ref{PropositionTrilinear}. Now assume that $\crm(\pi)=\Ocal_\F$ and write $\Phi : \otimes_v \pi_v\overset{\cong}{\longrightarrow}\pi$ for the isomorphism of $\GL_2(\Abb_\F)$-representations appearing implicitly in Section \ref{SectionInvariant}. For $v|\infty$, let $\Bscr(\pi_v)$ be an orthonormal basis of $\pi_v$ and $\Bscr(\pi_\infty)=\otimes_{v|\infty}\Bscr(\pi_v)$. If $v$ is a finite place and $v\neq\qmf$, the orthonormal basis $\Bscr(\pi_v,\Ocal_{\F_v})$ of the space of $\Krm_{v}$-invariant vector in $\pi_v$ only contains the spherical vector $\varphi_v^0$. If $v=\qmf$, then it is a classical result of Casselman \cite{casselman} that the space $\pi^{\Krm_{0,v}(\varpi_v)}$ has dimension $2$ and is generated by $\varphi_v^0$ and its translate by the matrix $\left( \begin{smallmatrix} 1 & \\ & \varpi_v \end{smallmatrix}\right)$, denoted by $\varphi_v^{0,\mmf_v}$. By \eqref{Comparition} and \eqref{CanonicalNorm}, a global basis $\Bscr(\pi,\qmf)$ can be constructed from the local ones through
$$\Bscr(\pi,\qmf)= \left\{ \frac{\Phi\left(\varphi_\infty\otimes\varphi_\qmf\otimes_{\substack{v<\infty \\ v\neq\qmf}}\varphi_{v}^0\right)}{\mathrm{N}(\pi)^{1/2}} \ : \ \varphi_\infty\in\Bscr(\pi_\infty), \ \varphi_\qmf \in \Bscr(\pi_\qmf,\mmf_\qmf)\right\},$$
with
$$
\mathrm{N}(\pi)=\left\{ \begin{array}{lcl}
2\Delta_\F^{1/2}\Lambda^*(\pi,\mathrm{Ad},1) & \ifm & \pi \ \mathrm{cuspidal} \\ 
 & & \\
\frac{\Delta_\F^{1/2}\Lambda^*(\pi,\mathrm{Ad},1)}{\Lambda_\F^*(1)} & \ifm & \pi \ \mathrm{Eisenstein}.
\end{array}\right.
$$
Hence $\Lscr(\pi,\qmf)$ can be rewritten as
\begin{alignat*}{1}
\Lscr(\pi,\qmf)= & \ \frac{1}{\mathrm{N}(\pi)}\sum_{\varphi_\infty \in\Bscr(\pi_\infty)} \\ \times & \sum_{\varphi_\qmf \in\Bscr(\pi_\qmf,\mmf_\qmf)} \left\langle \varphi_1^\qmf\varphi_2,\Phi\left(\varphi_\infty\otimes\varphi_\qmf\otimes_{ v\neq\qmf}\varphi_{v}^0\right)\right\rangle \overline{\left\langle \varphi_1\varphi_2^\qmf,\Phi\left(\varphi_\infty\otimes\varphi_\qmf\otimes_{ v\neq\qmf}\varphi_{v}^0\right)\right\rangle}.
\end{alignat*}
Fixing the infinite data, we see that the second line can be seen as
\begin{equation}\label{trilinear}
\sum_{\varphi_v\in \Bscr(\pi_v,\mmf_v)}\ell\left(\varphi_{1,v}^{\mmf_v}\otimes\varphi_{2,v}\otimes\varphi_v \right)\overline{\ell\left(\varphi_{1,v}\otimes\varphi_{2,v}^{\mmf_v}\otimes\varphi_v \right)},
\end{equation}
where $v=\qmf$ and $\ell$ is the $\GL_2(\F_v)$-invariant functional given by the composition of the inclusion $\pi_v\otimes\pi_{1,v}\otimes\pi_{2,v}\hookrightarrow \pi\otimes\pi_1\otimes\pi_2$ with the inner product. Now the result follows from the computations made in Section \ref{SectionOf}
\end{proof}
Using successively Propositions \ref{PropositionFact3} and \ref{PropositionIntegralRepresentation}, we obtain the factorization
$$\Lscr(\pi,\qmf)=\frac{C}{2\Delta_\F^{1/2}}\eta_\pi(\qmf)f(\pi_\infty)\frac{L(\pi\otimes\pi_1\otimes\pi_2,\tfrac{1}{2})}{\Lambda^*(\pi,\mathrm{Ad},1)}\prod_{v<\infty} \ell_v,$$
where the constant $C$ is given by
$$C= \left\{ \begin{array}{lcl}
 2\Lambda_\F(2) & \ifm & \pi,\pi_1,\pi_2 \ \mathrm{are \ cuspidal} \\
  & & \\
 2\Lambda_\F^*(1) & \ifm & \pi \ \mathrm{is \ Eisenstein \ and \ nonsingular} \\
  & & \\
 1 & \ifm & \pi \ \mathrm{is \ cuspidal \ and \ }\pi_2=1\boxplus 1,
\end{array}\right.$$
$\eta_\pi(\qmf)$ is the Atkin-Lehner eigenvalue at the place $v=\qmf$ ($\eta_\pi(\qmf)=1$ if $\crm(\pi)=\Ocal_\F$). According to \eqref{Fact} and \eqref{Fact3}, the local factors $\ell_v$ are given by
\begin{equation}\label{LocalFactorlv}
\ell_v = \left\{ \begin{array}{lcl}
1 & \ifm & v\neq \qmf \\ 
 & & \\
 \kappa & \ifm & v=\qmf \ \mathrm{and} \ \crm(\pi)=\Ocal_\F \\ 
  & & \\ 
 I_v(\varphi_{1,v}^{\mmf_v}\otimes\varphi_{2,v}\otimes\varphi_v)  & \ifm & v=\qmf \ \mathrm{and} \ \crm(\pi)=\qmf,
\end{array}\right.
\end{equation}
where $I_v$ is the normalized matrix coefficient defined in \eqref{DefinitionNormalizedMatrixCoefficient} and $\varphi_v$ denotes the new-vector in $\pi_v$ having norm $1$ with respect to $\langle \cdot,\cdot\rangle_v$. Finally, the factor at infinity is equal to
\begin{equation}\label{finfty}
f(\pi_\infty)= \sum_{\varphi_\infty \in \Bscr(\pi_\infty)} I_v(\varphi_{1,\infty}\otimes\varphi_{2,\infty}\otimes\varphi_{\infty})L(\pi_\infty\otimes\pi_{1,\infty}\otimes\pi_{2,\infty},\tfrac{1}{2}).
\end{equation}
It is not obvious but true that this Archimedean function is positive.

\subsection{A reciprocity formula} Let $\qmf$ and $\pmf$ two squarefree and coprime ideals of $\Ocal_\F$ of norms $q$ and $p$ respectively and let $\pi_1,\pi_2$ be two unitary automorphic representations which are unramified at all finite places and such that $\pi_1$ is cuspidal and $\pi_2$ is not isomorphic to a quadratic twist of $\pi_1$. For each unitary automorphic representation $\pi=\otimes\pi_v$ of conductor dividing $\qmf$, we define $\ell(\pi,\qmf)$ by multiplicativity over the finite places $\prod_{v \ \mathrm{finite}}\ell_v(\pi_v,\qmf_v)$, where $\ell_v$ is given in \eqref{LocalFactorlv}. Let $f$ be the non-negative function depending on the Archimedean type of $\pi$, $\pi_1$ and $\pi_2$ defined in \eqref{finfty}. We set $\eta_\pi(\qmf)$ for the Atkin-Lehner eigenvalue, $\lambda_\pi(\pmf)$ for the Hecke eigenvalue and $\Lambda^*(\pi,\mathrm{Ad},1)$ the first nonvanishing coefficient in the Laurent expansion around $s=1$ of the complete adjoint $L$-function.

\vspace{0.1cm}

We define the cuspidal part 
\begin{equation}\label{CuspidalPart}
\mathscr{C}(\pi_1,\pi_2,\qmf,\pmf):= C\sum_{\substack{\pi \ \mathrm{cuspidal} \\ \crm(\pi)|\qmf}}\lambda_\pi(\pmf)\eta_\pi(\qmf)\frac{L(\pi\otimes\pi_1\otimes\pi_2,\tfrac{1}{2})}{\Lambda(\pi,\mathrm{Ad},1)}f(\pi_\infty)\ell(\pi,\qmf).
\end{equation}
For the continuous part, we denote by $\pi_{\omega}(it)$ the principal series $\omega|\cdot|^{it}\boxplus \omegabar|\cdot|^{-it}.$ We then set 
\begin{equation}\label{ContinuousPart}
\begin{split}
\mathscr{E}(\pi_1,\qmf,\pmf):= \frac{C}{4\pi}\sum_{\substack{\omega\in\widehat{\F^\times\setminus \Abb_\F^{1}} \\ \crm(\omega)=\Ocal_\F}}\int_{-\infty}^\infty & \lambda_{\pi_\omega(it)}(\pmf)f(\pi_{\omega_\infty}(it))\ell(\pi_\omega(it),\qmf) \\ \times & \frac{L(\pi_1\otimes\pi_2\otimes\omega,\tfrac{1}{2}+it)L(\pi_1\otimes\pi_2\otimes\omegabar,\tfrac{1}{2}-it)}{\Lambda^*(\pi_\omega(it),\mathrm{Ad},1)} dt,
\end{split}
\end{equation}
and
\begin{equation}\label{DefinitionMoment}
\Mscr(\pi_1,\pi_2,\qmf,\pmf)=:\frac{p^{1/2}}{p+1}\left(\Cscr(\pi_1,\pi_2,\qmf,\pmf)+\mathscr{E}(\pi_1,\pi_2,\qmf,\pmf)\right).
\end{equation}
By Section \ref{SectionApplyParseval}, we have $\Mscr(\pi_1,\pi_2,\qmf,\pmf)=\mathscr{P}_{\qmf,\pmf}(\varphi_1,\varphi_2)$ and the symmetry property \eqref{Symmetry} yields :
\begin{theorem}\label{Theorem} Let $\qmf,\pmf$ two squarefree and coprime ideals of $\Ocal_\F$ and $\pi_1,\pi_2$ be two unitary automorphic representations of $\PGL_2(\Abb_\F)$ which are unramified at all finite places and such that $\pi_1$ is cuspidal and $\pi_2$ is not isomorphic to a quadratic twist of $\pi_1$. Let $\Mscr(\pi_1,\pi_2,\qmf,\pmf)$ be the moment defined in \eqref{DefinitionMoment}. Then
$$\Mscr(\pi_1,\pi_2,\qmf,\pmf)=\Mscr(\pi_1,\pi_2,\pmf,\qmf).$$
\end{theorem}

\subsection{A note on the Archimedean test function $f(\pi_\infty)$} Until here, we did not say anything about the function $f$ defined in \eqref{finfty} and appearing in the moments \eqref{CuspidalPart} and \eqref{ContinuousPart}; the only thing we know is that it is non-negative. For many applications in analytic theory of $L$-functions, it is fundamental that $f$ satisfies at least the following property : given $\pi=\pi_\infty\otimes\pi_{\mathrm{fin}}$ a unitary automorphic representation of $\PGL_2(\Abb_\F)$, there exists $\varphi_{i,\infty}\in\pi_{i,\infty}$, $i=1,2$ having norm $1$ and a basis $\Bscr(\pi_\infty)$ (recall that $f$ depends implicitly on these data) such that $f(\pi_\infty)$ is bounded below by a power of the Archimedean conductor $\crm_\infty(\pi).$ It is a result of Michel and Venkatesh \cite[Proposition 3.6.1]{subconvexity} that such a choice exists when for all $v | \infty$, either $\pi_{1,v}$ or $\pi_{2,v}$ is a principal series. In this case, we can obtain the lower bound 
\begin{equation}\label{LowerBound}
f(\pi_\infty)\geqslant \frac{C(\pi_{1,\infty},\pi_{2,\infty},\varepsilon)}{\crm_\infty(\pi)^{1+\varepsilon}},
\end{equation}
for any $\varepsilon>0$ and where $C(\pi_{1,\infty},\pi_{2,\infty},\varepsilon)$ is a positive constant.

\section{Local Computations}
Through this section, we fix $k$ a non-Archimedean local field of characteristic zero and we denote by $\Ocal$ the ring of integers, $\mmf$ the maximal ideal, $\varpi$ a uniformizer, i.e. a generator of $\mmf$ and let $q$ be the cardinality of the residue field $\Ocal/\mmf$. Let $\pi,\pi_1,\pi_2$ be irreducible, smooth admissible, infinite dimensional representations of $\GL_2(k)$ with trivial central character. We assume that $\pi_1$ and $\pi_2$ are unramified and $\pi$ has conductor $\mmf$ or $\Ocal$. In particular, both $\pi_1$ and $\pi_2$ are principal series and $\pi$ is either principal or special, depending on whether $\crm(\pi)=\Ocal$ or $\crm(\pi)=\mmf$ (see for example \cite[Remark 4.25]{gelbart}). We Fix an invariant inner product on the spaces underlying these representations together with an equivariant isometry with their respective Whittaker model endowed with the inner product \eqref{NormalizedInnerProduct}. We also denote by $\Bscr(\pi,\mmf)$ an orthonormal basis of the space of $\Krm_0(\varpi)$-invariant vectors in $\pi$. The goal of this section is to evaluate the local factor $\ell_v$ \eqref{LocalFactorlv} appearing in Section \ref{SectionApplyParseval}. Let $\varphi\in\pi$ and $\varphi_i\in\pi_i$ be the new-vectors having norm $1$.

\subsection{The case $\crm(\pi)=\mmf$} We evaluate here the matrix coefficient $I(\varphi\otimes\varphi_1^\mmf\otimes\varphi_2)$ defined in \eqref{DefinitionNormalizedMatrixCoefficient} as follows : since $\pi_{2}$ is unramified, we can realize $\pi_{2}$ in a principal series representation $\mu\boxplus \mu^{-1}$ with $\mu$ an unramified quasi-character of $k^\times$. Let $f_{2} \in \mu\boxplus \mu^{-1}$ be the new-vector corresponding to $\varphi_{2}$ and having norm $1$ with respect to our fix invariant inner product defined in Section \ref{SectionInvariant}. If $\Wcal(\pi)$ and $\Wcal(\pi_{1})$ denote the respective Whittaker models of $\pi$ and $\pi_{1}$ and $W$ (resp. $W_{1}$) is the Whittaker function associated to $\varphi$ (resp. to $\varphi_{1}$), then we can use \cite[Lemma 3.4.2]{subconvexity} to obtain
$$ I(\varphi,\varphi_1^{\mmf}\otimes\varphi_{2}) =\frac{|\Zcal(W,W_{1}^{\mmf},f_{2})|^2}{L(\pi\otimes\pi_{1}\otimes\pi_{2},\tfrac{1}{2})},$$
where the zeta integral is given by
$$\Zcal(W,W_{1}^{\mmf},f_{2}) = \int_{\Nrm(k)\setminus \PGL_2(k)}WW_{1}^{\mmf}f_{2}.$$
Since $\pi$ has conductor $\mmf$, a direct computation shows that $|W(1)|^2=\zeta_{k}(1)/\zeta_{k}(2)$ (use the fact that $W$ has norm 1 with respect to \eqref{NormalizedInnerProduct} and \cite[(11.14)]{sparse}). The zeta integral is computed explicitely in \cite[Lemma 11.6]{sparse}; the author found
$$\frac{|\Zcal(W,W_{1}^{\mmf},f_{2})|^2}{L(\pi\otimes\pi_{1}\otimes\pi_{2},\tfrac{1}{2})}=\frac{\zeta_{k}(1)}{\zeta_{k}(2)}\frac{q}{(q+1)^2}.$$

\subsection{The case $\crm(\pi)=\Ocal$}\label{SectionOf} According to \eqref{trilinear} in the proof of Proposition \ref{PropositionFact3}, we need to show that if $\ell : \pi\otimes\pi_1\otimes\pi_2 \rightarrow \Cbb$ is any invariant functional, then there exists a constant $\kappa$ such that
\begin{equation}\label{True}
\sum_{v\in\Bscr(\pi,\mmf)}\ell(v\otimes\varphi_1^\mmf\otimes\varphi_2)\overline{\ell(v\otimes\varphi_1\otimes\varphi_2^\mmf)}=\kappa|\ell(\varphi\otimes\varphi_1\otimes\varphi_2)|^2.
\end{equation}
For this, we shall determine explicitly the orthonormal basis $\Bscr(\pi,\mmf)$. It is a famous result of Casselman \cite{casselman} that the space of $\Krm_0(\varpi)$-invariant vectors has dimension two and is generated by $\varphi$ and $\varphi^\mmf$. Using exaclty the same ideas as in Proposition \ref{PropositionTrilinear}, we obtain by \eqref{PropHecke}
\begin{equation}\label{DefinitionKappa}
\langle \varphi,\varphi^\mmf\rangle = \frac{q^{1/2}}{q+1}\lambda_\pi=: \kappa(\pi),
\end{equation}
where for any principal series representation $\tau = \mu_1\boxplus\mu_2$, we write $\lambda_\tau = \mu_1(\varpi)+\mu_2(\varpi)$. Hence we can take
\begin{equation}\label{OrthonormalBasis}
\Bscr(\pi,\mmf)=\left\{ \varphi,\frac{\varphi^\mmf-\kappa(\pi)\varphi}{\sqrt{1-\kappa(\pi)^2}}\right\}.
\end{equation}
Setting $\varsigma=1-\kappa(\pi)^2$, we obtain with the choice \eqref{OrthonormalBasis}
\begin{alignat*}{1}
\sum_{v\in\Bscr(\pi,\mmf)}\ell(v\otimes\varphi_1^\mmf\otimes\varphi_2)\overline{\ell(v\otimes\varphi_1\otimes\varphi_2^\mmf)} = & \ \left(1+\frac{\kappa(\pi)^2}{\varsigma}\right)\ell(\varphi\otimes\varphi_1^\mmf\otimes\varphi_2)\overline{\ell(\varphi\otimes\varphi_1\otimes\varphi_2^\mmf)} \\ 
+ & \ \frac{1}{\varsigma}\ell(\varphi^\mmf\otimes\varphi_1^\mmf\otimes\varphi_2)\overline{\ell(\varphi^\mmf\otimes\varphi_1\otimes\varphi_2^\mmf)} \\
 - & \ \frac{\kappa(\pi)}{\varsigma}\ell(\varphi^\mmf\otimes\varphi_1^\mmf\otimes\varphi_2)\ell(\varphi\otimes\varphi_1\otimes\varphi_2^\mmf) \\
 - & \ \frac{\kappa(\pi)}{\varsigma}\ell(\varphi\otimes\varphi_1^\mmf\otimes\varphi_2)\ell(\varphi^\mmf\otimes\varphi_1\otimes\varphi_2^\mmf).
\end{alignat*}
Using again \eqref{PropHecke}, we see that
$$\ell(\varphi\otimes\varphi_1^\mmf\otimes\varphi_2)=\ell(\varphi^\mmf\otimes\varphi_1\otimes\varphi_2^\mmf)=\kappa(\pi_1)\ell(\varphi\otimes\varphi_1\otimes\varphi_2)$$
and
$$\ell(\varphi\otimes\varphi_1\otimes\varphi_2^\mmf)=\ell(\varphi^\mmf\otimes\varphi_1^\mmf\otimes\varphi_2)=\kappa(\pi_2)\ell(\varphi\otimes\varphi_1\otimes\varphi_2).$$
where we recall that $\kappa(\pi_i)$ is related to the Hecke eigenvalue via \eqref{DefinitionKappa}. Hence \eqref{True} is true with 
\begin{equation}\label{ValueKappa}
\kappa = -\frac{\kappa(\pi)}{\varsigma}\left(\kappa(\pi_1)^2+\kappa(\pi_2)^2\right)+2\frac{\kappa(\pi_1)\kappa(\pi_2)}{\varsigma}.
\end{equation}

\bibliography{Reciprocity}
\bibliographystyle{alpha}

\end{document}